\RequirePackage{fix-cm}
\documentclass[smallextended]{svjour3}       
\usepackage{graphicx}
\usepackage{enumerate}
\usepackage{latexsym,amssymb,amsxtra}
\usepackage{amsmath}
\usepackage{stmaryrd}
\usepackage{pstricks}
\usepackage{pst-node}
\usepackage{multido}
\usepackage{mathrsfs}
\usepackage{tikz}
\usepackage{pgf}

\def\/{\, | \,}

\def\ind{{\mathchoice {\rm 1\mskip-4mu l} {\rm 1\mskip-4mu l}
{\rm 1\mskip-4.5mu l} {\rm 1\mskip-5mu l}}}

\def\N{{\mathbb N}}

\def\Z{{\mathbb Z}}

\def\P{{\mathbb P}}

\newcommand{\R}{{\mathbb R}}

\newcommand{\pae}[1]{\mbox{$\lfloor \kern-1pt #1 \kern-1pt \rfloor$}}
\newcommand{\paep}[1]{\mbox{$\lceil \kern-1pt #1 \kern-1pt \rceil$}}

\newcommand\bp{\bar{\mathbf P}}

\newcommand\pr[1]{{\mathbb P}\left[#1\right]}

\newcommand\suite[1]{\left\{#1_n;\,n\in\N\right\}}

\newcommand\suiten[1]{\left\{#1;\,n\in\N\right\}}
\newcommand\om{{\omega}}
\newcommand\gre{\textbf{e}}
\newcommand\maG{{\mathcal G}}

\newcommand\maF{{\mathcal F}}
\newcommand\maB{{\mathcal B}}
\newcommand\maA{{\mathcal A}}

\newcommand\maC{{\mathcal C}}
\newcommand\maE{{\mathcal E}}
\newcommand\maD{{\mathcal D}}

\def\RS{\overline {(\R_+)^S}}
\def\lllbracket{\left[}
\def\rrrbracket{\right]}
\begin{document}

\title{Coupling in the queue with impatience: case of several servers}


\author{Pascal Moyal}


\institute{P. Moyal \at
Universit\'e de Technologie de Compi\`egne\\
Laboratoire de Mathématiques Appliqu\'ees de Compi\`egne\\
Rue Roger Couttolenc
CS 60319\\
60203 Compi\`egne, France.
              Tel.: +33 3 44 23 44 99\\
              \email{pascal.moyal@utc.fr}           
}

\date{Received: date / Accepted: date}

\maketitle

\begin{abstract}
We present the explicit construction of a stable queue with several 
servers and impatient customers, under stationary ergodic
assumptions. Using a stochastic comparison of the (multivariate) workload sequence with two monotonic stochastic recursions, we propose a sufficient condition of existence of a
unique stationary state of the system using Renovation theory. Whenever this condition is relaxed we use extension techniques
to prove the existence of a stationary state in some cases.
\end{abstract}

\section{Introduction}
\label{sec:intro}
In the catalog of the various probabilistic techniques for studying the stability of discrete-event random systems, the so-called stationary ergodic framework
has two main features~: (i) it addresses general statistical assumptions (stationarity and ergodicity, but not necessarily independence, of the random sequences representing the input) and (ii) it allows an {\em explicit} construction of the steady state.
Specifically, we know since the pioneering works of Loynes \cite{Loynes62} and then Borovkov \cite{Bor84} that, when assuming that the input to the system is time-stationary and ergodic, 
backwards schemes and strong backwards coupling convergence, can lead, via the resolution on an adequate probability space, of a pathwise fix point equation of the form (\ref{eq:recurstat}) below, 
to an explicit construction of the stationary state of the system under consideration. 
Its existence/uniqueness, which is in turn equivalent to the existence/uniqueness of a solution to (\ref{eq:recurstat}), can be investigated
by studying the algebraic properties of the driving map of the stochastic recursion representing the model (denoted by $\varphi$ in (\ref{eq:recurstat})). The details of the construction can be found {\em e.g.} in \cite{Bor84}, \cite{BranFranLis90}, Chapters 1 and 2 of \cite{BacBre02}, or Chapter 2 of \cite{DecMoy12}.

\medskip

Beyond the derivation of the stability region itself, this explicit construction can be used for various purposes: one can e.g. use pathwise representations to compare systems in steady state, 
via a stochastic recursive representation of the system, and the stochastic ordering of a given performance metric. The reader is referred to Chapter 4 of \cite{BacBre02}, in which many such comparison results for queuing systems 
(along service disciplines, distributional primitives, etc.) can be found. 
Specific comparison results concerning single-server queueing systems with impatient customers are also provided in \cite{Moy08} and \cite{Moy13}. 
Moreover, it is well known since the seminal work of Propp and Wilson \cite{PW96}, that coupling from the past algorithms, which 
mostly use backwards coupling convergence, provide a powerful tool for simulating in many cases (monotonicity, stochastic bounds of Markov chains) the steady state of the system. 

\medskip

Along these lines, this paper is devoted to the explicit construction of a stationary queue with $S$ servers ($S \ge 1$) and impatient customers, by a scheme {\em \`a la} Loynes.  
Models with impatience (or abandonment, reneging) have been introduced in the queueing literature to represent a strong real-time constraint on the system: the requests have a due date,
before which their treatment must be initiated, or completed. Specifically, we assume hereafter that any incoming customer is either served if a server becomes available before her deadline,
or else eliminated forever once the deadline has elapsed.
Observe that a loss system of $S$ systems (i.e. there is no waiting room, so the incoming customers
are either served provided that a server is immediately available, or immediately lost otherwise) is a particular case of the
present model, for an identically null patience.

It is intuitively clear that a queueing system in which the sojourn times of the customers are bounded by their patience, cannot explode in the long run.
However, under such general assumptions, due to the intrinsic non-monotonicity of this non-conservative model, the existence, and uniqueness, of the stationary state, do not follow from the above observation.
The first contributions to the stability study of queues with impatience can be found in \cite{MR86d:60103} and \cite{HebBac81}, in which a GI/GI/1/1+GI-FIFO queue (that is, all input sequences: inter-arrivals, service times and patience times are IID and there is a single server) is investigated, as well as a muti-server M/M/S/S+GI-FCFS queue. For the single-server case, the independence assumption is relaxed in 
\cite{Moy09}, which proposes a constructive scheme of the stationary workload using renovation theory. 
Under more general statistical assumptions on the input, \cite{Moy15} proposes an explicit extension scheme allowing to
construct the stationary workload at least on a enriched probability space, in line with the {\em weak stability} approach of \cite{Lis82,Neu83,Fli83}. For other references on the stability study of non-monotonic queueing systems, let us also mention \cite{Fli88,Moy16} on loss systems with several servers,
\cite{Moy13} on single-server queues with impatience and general service discipline, and \cite{AK97,AK99} on perturbed sub-critical single-server queues.

\medskip

The present paper addresses the case of a G/G/S/S+G-FCFS queue: there are $S$ servers obeying the First Come, First Served (FCFS) rule to serve impatient customers, 
and the sequences of inter-arrival times, service times and patience times of the customers are assumed stationary and ergodic, but not necessarily independent.
In that case, the workload vector measured at arrival times is stochastic recursive, see Section \ref{subsec:workload} below.
By constructing the suitable stochastic bounds of that recursion (respectively driven by the random maps
defined by (\ref{eq:defoverphi}) and (\ref{eq:defunderphi}) below),
we provide explicit conditions for the existence and uniqueness of a stationary workload vector, by applying under specific conditions, the Renovation theory
of Borovkov and Foss \cite{Bor84,Foss92}. Whenever no renovation of the stochastic upper-bound is granted, we use
the tools of weak stability, to solve the stability problem at least on an extension of the original probability space.

By doing so, our work generalizes:
\begin{itemize}
\item The stability results of \cite{MR86d:60103} and \cite{HebBac81} to non-exponential, and non-necessarily independent inter-arrival, service and patience times;
\item The results of \cite{Moy09} and Section 5 of \cite{Moy15} to an arbitrary number of servers;
\item The results of \cite{Lis82,Fli83,Fli88} and \cite{Moy16} (Theorem 2), to non-necessarily null patience times.
\end{itemize}

\bigskip


\section{Preliminary}
\label{sec:prelim}

\subsection{Main notation}
\label{subsec:notation}
In what follows, $\R$ denotes the real line and $\R_+$, the subset of non-negative numbers. Denote $\mathbb N$ (respectively $\N^*$, $\mathbb Z$), the subset of non-negative (resp. positive, relative) integers. For any  $p$ and $q$ in $\N$, $\llbracket p,q \rrbracket$ denotes the family
$\{p,p+1,...,q\}$. For any $x,y \in \R$, denote $x\wedge y=\min(x,y)$, $x \vee y=\max(x,y)$ and $x^+ = x \vee 0$.
Let $S \in \N^*$. We denote for all $u, v \in \R^S$ and $\lambda \in \R$,
\begin{align*}
u&=\left(u(1),u(2),...,u(S)\right);\quad\quad
\lambda u=\left(\lambda u(1),...,\lambda u(S)\right);\\
u+v &=\left(u(1)+v(1),...,u(S)+v(S)\right);\quad\quad u^+ =(u(1)^+,u(2)^+,...,u(S)^+);\\
\bar u &,\,\mbox{ the fully ordered version of }u,\mbox{ i.e. }\bar u(1) \le ... \le \bar u(S);\\
\mathbf 0&=(0,\,...,\,0);\,
\quad\quad\gre_i =(0,\,...,\,\underbrace{1}_{i},\, ... 0)\mbox{ for all }i\in\llbracket 1,S \rrbracket.
\end{align*}
Denote by $\overline {(\R_+)^S}$ the subset of fully ordered vectors, furnished with the euclidean norm.
We equip $\RS$ with the coordinate-wise ordering "$\prec$", {\em i.e.}
$u \prec v$ if and only if $u(i) \le v(i)$ for all $i\in \llbracket 1,S \rrbracket.$
For two vectors $u,v \in \RS$, with some abuse of notation we denote
\begin{equation}
\label{eq:lll}
\lllbracket u,v \rrrbracket = \left\{w \in \RS;\,u \prec w \prec v\right\}.
\end{equation}

\subsection{Stochastic Recursive Sequences}
\label{subsec:SRS}
The stability results presented below will be stated in ergodic theoretical terms, by investigating the existence of a stationary
version of a stochastic recursion representing the system under consideration. Let us recall the basics of this framework.
For more details, the reader is referred to the monographs \cite{BranFranLis90},
\cite{BacBre02} (Sections 2.1 and 2.5), \cite{Rob03} (Chapter 7), or \cite{DecMoy12} (Sections 3.1 and 3.2).

\paragraph{Stationary ergodic quadruple}
We say that $\mathscr Q=\left(\Omega,\mathscr F,\mathbb P,\theta\right)$ is a stationary ergodic quadruple if
$\left(\Omega,\mathscr F,\mathbb P\right)$ is a probability space and $\theta$ is a bijective operator
on $\mathscr F$ (we denote by $\theta^{-1}$ its reciprocal operator), such that for any $\maA \in \mathscr F$, $\pr{\maA}=\pr{\theta^{-1}\maA}$ and any $\theta$-invariant event
$\maB$ ({\em i.e.}, such that $\maB=\theta^{-1}\maB$) is either negligible or almost sure. Note that this is then also
true for any $\theta$-contracting event $\maB$ ({\em i.e.} such that $\pr{\maB \,\Delta \,\theta^{-1}\maB}=0$, where $\Delta$ denotes the symmetrical difference). We denote for any $n\in\N^*$, by
$\theta^n$ (respectively, $\theta^{-n}$), the $n$-th iterate of $\theta$ (resp., of $\theta^{-1}$).
Hereafter, for any random variable (r.v., for short) defined on $\mathscr Q$, we denote by $X[\om]$ the value of $X$ at the sample $\om \in \Omega$.

\paragraph{Stochastic recursive sequences}
In many concrete cases, the evolution of a physical discrete-event random system defined on a reference probability space
$\left(\tilde\Omega,\tilde{\mathscr F},\tilde{\mathbb P}\right)$ can be represented recursively
by a random sequence $\suite{\tilde{W}}$ valued in some Polish space $E$ (furnished with its Borel $\sigma$-algebra), and satisfying the recursion
\[\tilde W_{n+1} =f\left(\tilde W_n,\tilde \alpha_n\right),\,\tilde{\mathbb P}\mbox{-a.s.}\]
for some $F$-valued random sequence $\suite{\tilde \alpha}$ (where $F$ is some auxiliary space), and $f$ a measurable map from $E \times F$ to $E$.
Then, whenever $\suite{\tilde \alpha}$ is identically distributed and ergodic, one can construct a stationary ergodic quadruple $\mathscr Q=\left(\Omega,\mathscr F,\mathbb P,\theta\right)$, on the bi-infinite canonical space $\Omega:=F^{\mathbb Z}$ of $\suite{\tilde \alpha}$.
In the queueing context, $\suite{\tilde \alpha}$ typically represents the input to the queue (in the present case: inter-arrival, service and patience times).
The quadruple $\mathscr Q$ is then often called {\em Palm space} of arrivals.
Moreover, there exists on
$\mathscr Q$ a $F$-valued r.v. $\alpha$, and a $E$-valued r.v. $X$ such that the recursion
\begin{equation}
\label{eq:defSRS}
\left\{\begin{array}{ll}
W^{[X]}_{0} &=X;\\
W^{[X]}_{n+1} &=f\left(W^{[X]}_{n},\alpha\circ\theta^n\right),\,n\in\N,
\end{array}\right.
\end{equation}
has the same distribution on $\Omega$ as that of $\tilde W_n$ on $\tilde \Omega$, provided that $\tilde W_0$ and $X$ have the same law. 
Let $\mathcal M(E)$ be the set of measurable mappings from $E$ to itself, and define the $\mathcal M(E)$-valued r.v. $\varphi$ by 
\[\varphi:\left\{\begin{array}{ll}
E &\to E\\
x &\mapsto f(x,\alpha),
\end{array}\right.\]
we say that the sequence $\suite{W^{[X]}}$ is the {\em stochastic recursion} initiated by $X$ and driven by $\varphi$.
Moreover, it is easiy checked that, if the r.v. $X$ satisfies the equation
\begin{equation}
\label{eq:recurstat}
X\circ\theta=\varphi\left(X\right),\,\mathbb P-\mbox{ almost surely,}
\end{equation}
then the sequence $\suiten{X\circ\theta^n}$ (which is stationary by definition) satisfies the recursion (\ref{eq:defSRS}).

In other words, we have $W^{[X]}_n=X\circ\theta^n$ for all $n$, $\mathbb P$-a.s., which, in view of the above remark, entails that
$\suiten{X\circ\theta^n}$ coincides in distribution with a time-stationary version of the sequence $\suite{\tilde W}$.
To summarize, to each solution $X$ of (\ref{eq:recurstat}) corresponds a unique stationary distribution for the sequence $\suite{\tilde W}$ on the
original probability space $\tilde\Omega$.

\medskip

As a conclusion, investigating the existence, and possibly uniqueness of a stationary distribution for $\suite{\tilde W}$ on the original space,
amounts to investigating the existence and uniqueness of a solution to the pathwise equation (\ref{eq:recurstat}) on the Palm space.
This is the framework we adopt in this paper. As opposed, for instance, to the paradigm of Markov chains, this will allow, first, to address
the stability problem under more general assumptions (stationarity but not necessarily independence of the input) and second, to use the very properties of the map $\varphi$ (monotonicity and/or pathwise orderings between maps) to solve (\ref{eq:recurstat}).

\paragraph{Weak resolution}
The two main tools existing in the literature for solving (\ref{eq:recurstat}) are (i) Loynes's construction, in the case where $\varphi$ enjoys some monotonicity property
(see \cite{Loynes62} and Chapter 2 of \cite{BacBre02}) and (ii)
Borovkov and Foss's Theory of Renovating events \cite{Bor84,Foss92}.
They will be the first two pillars of our analysis. However, under the most general assumptions we will encounter cases in which none of the above applies. In such cases, we will resort to the extension techniques developed in \cite{Moy15} (which is closely related to the construction in \cite{Lis82,Neu83,Fli83} for the loss system, and to the
alternative extension technique of \cite{AK97,Moy15}) to solve (\ref{eq:recurstat}), at least in a weaker sense which is formalized as follows,
\begin{definition}
Let $\mathscr Q=\left(\Omega,\maF,\mathbf P,\theta\right)$ be a stationary ergodic
quadruple. We say that there exists a {\em weak
solution} to (\ref{eq:recurstat}) whenever there exists a
stationary extension $\left(\bar\Omega, \bar\maF, \bar{\mathbf P},
\bar\theta\right)$ of $\mathscr Q$,
such that
\begin{itemize}
\item $\bp$ is a $\bar \theta$-invariant probability
on $\bar \Omega$ having $\Omega$-marginal $\mathbf P$,
\item there exists a $E\times \mathcal M(E)$-valued r.v. $\left(\bar Z,\bar \varphi\right)$, such that the $\Omega$-marginal of $\bar\varphi$ is the distribution of $\varphi$, and such that
$$\bar Z\circ\bar\theta=\bar\varphi\left(\bar Z\right), \bp-\mbox{a.s.}.$$
\end{itemize}
\end{definition}
This approach proves particularly useful in the case where it is possible to "project" the solution $(\bar X,\bar\varphi)$ in the original space, 
to construct a solution to (\ref{eq:recurstat}) on $\mathscr Q$.
This will be the case in Theorem \ref{thm:GIGI} below.

\section{FCFS system with $S$ servers and impatient customers}
\label{sec:model}

\subsection{The model}
\label{subsec:model} Following the notation of Section
\ref{subsec:SRS}, the input of the system is represented on a
reference probability space $\left(\tilde\Omega,\tilde{\mathscr
F},\tilde{\mathbb P}\right)$ by the sequence
$\suite{\tilde\alpha}=\suiten{\left(\tilde\tau_n,\tilde\sigma_n,\tilde D_n\right)}$,
where for any $n$, $\tilde\tau_n$, $\tilde\sigma_n$ and $\tilde D_n$ denote
respectively the inter-arrival time after the arrival of the $n$-th
customer (denoted by $C_n$), the service time and the patience time of $C_n$, all measured in a
fixed time unit. The customers enter a system with $S$ equivalent servers working without vacations.
There is a single line, in which customer $C_n$ is put if all servers are busy upon arrival. Otherwise, $C_n$ is attended by a free server, chosen arbitrarily among all
such servers.
The servers obey the First Come, First Served (FCFS) rule to chose a customer in line upon a service completion.
For any $n$, $C_n$ is furthermore impatient: she agrees to wait in line only for the duration of her patience $\tilde D_n$.
If she does not reach a server by that time, she leaves the system forever.
To the contrary, if she is able to start service before the end of her patience, she will let her service proceed until completion no matter what, even if her patience elapses during service. The
sequences $\suite{\tilde\tau}$, $\suite{\tilde\sigma}$  and $\suite{\tilde D}$ are all
assumed identically distributed and ergodic. In other words, according to Barrer's notation the system we consider is G/G/S/S+G-FCFS. 
We also assume that
$\tilde\tau_0$, $\tilde\sigma_0$ and $\tilde D_0$ are integrable, and that
$\tilde{\mathbb P}\left(\tilde\tau_0>0\right)=1$, in other words the
arrival process is simple.

As described in Section \ref{subsec:SRS}, we represent the system on the Palm space of arrivals
$\mathscr Q=\left(\Omega,\mathcal F, \mathbb P, \theta\right)$ associated to $\suiten{\left(\tilde\tau_n,\tilde\sigma_n,\tilde D_n\right)}$.
Time 0 is set at the arrival time of a customer denoted $C_0$ (see again
\cite{BacBre02,BranFranLis90}), and $C_0$ requests a service time of duration $\sigma$ and has patience time $D$. The following customer $C_1$ enters the system at time $\tau$. Then, for any $n \in \Z$, $\sigma\circ\theta^n$ and $D\circ\theta^n$ are respectively interpreted as the service and patience time of customer $C_n$, and $\tau\circ\theta^n$ represents the time epoch between the arrivals of customers $C_n$ and $C_{n+1}$. It follows from our very assumptions that the shift $\theta$ is $\mathbb P$-stationary and ergodic (so $\mathscr Q$ is a stationary ergodic quadruple), that $\sigma$, $\tau$ and $D$ are integrable, and that $\pr{\tau>0}=1$.

\subsection{Workload vector}
\label{subsec:workload}
As we now demonstrate, the FCFS discipline allows one to represent the system by a simple $S$-dimensional stochastic recursion:
first observe that for any $n$, the identity of the server which will eventually attend $C_n$ is in fact determined upon her arrival (say, at time $T_n$), by the input up to $T_n$. We prove this statement by induction on $n$. Suppose that just before $T_n$ (i.e. just before the arrival of $C_n$), we know the residual service times of all customers in service, and the service times of all customers in line, ranked in increasing order. We denote by
\[\tilde{\varsigma}_n^0=\left(\tilde{\varsigma}_n^0(1),...,\tilde{\varsigma}_n^0(S)\right) \in \RS,\,\]
the ordered vector of residual service times. In other words,
$\tilde{\varsigma}_n^0(i)$ represented the $i$-th smallest residual service time of a customer in service, if any, or 0 if there are not more than $S-i$ busy servers at $T_n$.
For the time being, we call {\em server } $i$, $i\in\llbracket 1,S \rrbracket$, the server whose residual service time
at $T_n$ is given by $\tilde{\varsigma}^0_n(i)$.
Denote also by $Q_n$ the number of customer in line just before $T_n$ (i.e., not including $C_n$), and if $Q_n>0$, by $i^1_n\le i^2_n \le ... \le i^{Q_n}_n$, the indexes of the customers in line at $T_n$, if any, ranked in the FCFS order.
Then, according to FCFS,
\begin{itemize}
\item if $\tilde{\varsigma}^0_n(1) \le \tilde D_{i^1_n}$, then customer $C_{i^1_n}$ will be attended by the server of residual service time $\tilde{\varsigma}^0_n(1)$.
If not, she will necessarily be lost before service. Then, denote by
\[\tilde{\varsigma}_n^1=\overline{\tilde{\varsigma}^0_n+\tilde \sigma_{i^1_n}\ind_{\{\tilde{\varsigma}^0_n(1) \le \tilde D_{i^1_n}\}}\gre_1},\]
the vector of ordered virtual workloads of the $S$ servers, taking into account the customers in service and the first customer in line;
\item if $\tilde{\varsigma}^1_n(1) \le \tilde D_{i^2_n}$, then customer $C_{i^2_n}$ will be attended by the server of virtual workload $\tilde{\varsigma}^1_n(1)$.
If not $C_{i^2_n}$ will be lost. Then, denote by
\[\tilde{\varsigma}_n^{2}=\overline{\tilde{\varsigma}_n^1+\tilde \sigma_{i^2_n}\ind_{\{\tilde{\varsigma}^1_n(1) \le \tilde D_{i^2_n}\}}\gre_1},\]
the vector of ordered virtual workloads, taking into account the customers in service and the first two customer in line;
\item $\vdots$
\item On and on, we can construct inductively over $\llbracket 1,Q_n \rrbracket$, the vector $\tilde W_n:=\tilde{\varsigma}_n^{Q_n}$ of ordered virtual workloads, taking into account the residual service times and initial service times of the customers who wait in line just before $T_n$.
    This vector is simply called {\em workload vector} at $T_n$. Observe that $\tilde W_n$ only adds the workloads brought by the customers present in the system just before $T_n$, and {\em who will eventually be served}, and not the ones who will be discarded before service, which are not seen by this descriptor.
\end{itemize}

From the FCFS assumption, customer $C_n$ is served (and brings the workload $\tilde \sigma_n$ to the corresponding server) if and only if
$\tilde W_n(1)=\tilde{\varsigma}_n^{Q_n}(1) \le \tilde D_n$. Then, a time duration $\tilde \tau_n$ elapses before the arrival of $C_{n+1}$. It is clear that we can retrieve from
$\tilde{\varsigma}_n^0$ and from the service times of customers $C_{i^1_n},...,C_{i^{Q_n}_n}$, the indexes of the customers entering service and the indexes of the servers attending them between $\tilde T_n$ and $\tilde T_{n+1}=\tilde T_n+\tilde \sigma_n$, if any, and thus, the values of all residual service times and service times of customers in service just before $\tilde T_{n+1}$, and reiterate the same procedure. Then, denoting by $\tilde W_{n+1}$ the workload vector thereby obtained, we retrieve that
\begin{equation}
\label{eq:recurWn}
\tilde W_{n+1} = \overline{\left[\tilde W_n +\tilde \sigma_n\ind_{\{\tilde W_n(1) \le \tilde D_n\}}\gre_1 - \tilde\tau_n\mathbf 1\right]^+}, 
\end{equation}
which shows that the workload vector sequence $\suite{\tilde W}$ is stochastically recursive.

To summarize, an exhaustive representation of the system just before $\tilde T_n$ is given by the vector of
random length
\[\tilde \mu_n :=\left(\tilde{\varsigma}_n^0(1),...,\tilde{\varsigma}_n^0(S),\tilde \sigma_{i^1_n},....,\tilde \sigma_{i^{Q_n}_n},\tilde D_{i^1_n},....,\tilde D_{i^{Q_n}_n}\right),\]
a descriptor that is used in \cite{Moy13} to obtain a recursive representation under an arbitrary service discipline, for a single-server. However, in the particular case of FCFS, 
$\tilde W_n$ provides a more simple $S$-dimensional recursion. Let us also notice that the relation (\ref{eq:recurWn}) generalizes Eq. (2.1) in \cite{HebBac81} to several servers. 
It can also be seen as an extension of Kiefer and Wolfowitz recursion \cite{KW55} to the case of impatient customers. In other words, in the particular case of FCFS the model amounts to a system of $S$ {\em parallel} servers, each having its own line, in which the entering customers are always sent to the line of the server having the least remaining workload, provided that the latter does not exceed its patience. 

In view of the construction of Section \ref{sec:prelim} and from (\ref{eq:recurWn}), the existence (resp., uniqueness) of a stationary distribution for the sequence $\suite{\tilde W}$ amounts to the existence (resp., uniqueness), on the quadruple $\mathscr Q$, of a solution to the equation
\begin{equation}
\label{eq:recurstatImp} W\circ\theta=\varphi(W),\,\P-\mbox{ a.s.,}
\end{equation}
where the random map $\varphi$ is defined by
\[\varphi:\left\{\begin{array}{lll}
\RS &\longrightarrow &\RS\\
                  u &\longmapsto & \overline{\left[u + \sigma\ind_{\{u(1)\le D\}}\gre_1-\tau\mathbf 1\right]^+},
                  \end{array}\right.\]
or in other words, for all $u \in \overline{\left(\R_+\right)^S}$,
\begin{equation}
\label{eq:coordSRSImp}
\left\{
\begin{array}{ll}
\varphi(u)(i) = \left[\left(u(i) \vee \left(u(1)+\sigma\ind_{\{u(1)\le D\}}\right)\right)\wedge u(i+1)-\tau\right]^+,\,i\in \llbracket 1,S-1\rrbracket;\\
\varphi(u)(S) = \left[u(S) \vee \left(u(1)+\sigma\ind_{\{u(1) \le D\}}\right)-\tau\right]^+.\\
\end{array}\right.
\end{equation}
This paper is devoted to the resolution of (\ref{eq:recurstatImp}).

\section{Stability results}
\label{sec:stab}

\subsection{Preliminary}
\label{subsec:prelim}
Let $S\ge 2$. Define on $\mathscr Q$ the following random maps,
\begin{equation}\label{eq:defoverphi}
\overline\varphi:\left\{\begin{array}{lll}
\overline{\left(\R_+\right)^S} &\longrightarrow &\overline{\left(\R_+\right)^S}\\
                  u &\longmapsto &v \,\,\mbox{ such that }\\
                & &v(j)=\biggl[\Bigl( u(j)\vee (\sigma+D)\Bigl) \wedge\, u(j+1)-\tau\biggl]^+,\,\,j\in \llbracket 1,S-1 \rrbracket;\\
                & &v(S)=\biggl[\Bigl( u(S)\vee (\sigma+D)\Bigl)
                -\tau\biggl]^+,
                  \end{array}
                  \right.\end{equation}
\begin{equation}
\label{eq:defunderphi}
\underline\varphi:\left\{\begin{array}{lll}
\overline{\left(\R_+\right)^S} &\longrightarrow &\overline{\left(\R_+\right)^S}\\
                  u &\longmapsto &v \,\,\mbox{ such that }\\
                & &v(j)=\biggl[\Bigl( u(j)\vee (\sigma\wedge D)\Bigl) \wedge\, u(j+1)-\tau\biggl]^+,\,\,j\in \llbracket 1,S-1 \rrbracket;\\
                & &v(S)=\biggl[\Bigl( u(S)\vee (\sigma\wedge D)\Bigl)
                -\tau\biggl]^+.
                  \end{array}
                  \right.\end{equation}
Let us now define the two following families of random variables,
\begin{align*}
\overline Z_\ell &=\left[\sup_{k\ge
\ell}\left((\sigma+D)\circ\theta^{-k} -\sum_{i=1}^k
\tau\circ\theta^{-i}\right) \right]^+,\, \ell \in \llbracket 1,S \rrbracket;\\
\underline Z_\ell &=\left[\sup_{k\ge
\ell}\left((\sigma\wedge D)\circ\theta^{-k} -\sum_{i=1}^k
\tau\circ\theta^{-i}\right) \right]^+,\, \ell \in \llbracket 1,S \rrbracket.
\end{align*}
Observe (see {\em e.g.} Exercise 2.6.1 in \cite{BacBre02}) that under the ongoing assumptions the $\underline Z_i$, and $\overline Z_i$, $i\in \llbracket 1,S \rrbracket$ are all $\mathbb{P}$-almost sure, from Birkhoff's Theorem. Define the $S$-dimensional random vectors
\begin{align}
\label{eq:defZvector} \overline Z&=\left(\overline Z_S,\overline Z_{S-1},...,\overline Z_1\right);\\
\label{eq:defunderZvector} \underline Z&=\left(\underline Z_S, \underline Z_{S-1},...,\underline Z_1\right).
\end{align}
The following two results gather Corollary 1 and Lemma 4 of
\cite{Moy16}, when replacing $\sigma$ by $\sigma+D$, and then $\sigma$ by $\sigma\wedge D$.
\begin{proposition}
\label{prop:stabover} The recursive equation
\begin{equation}
\label{eq:recurstatover} \overline Y\circ\theta=\overline\varphi\left(\overline Y\right),\,\mbox{ a.s.}
\end{equation}
admits at least one $\overline{\left(\R_+\right)^S}$-valued
solution. Furthermore, any such solution $\overline Y$ is such that
\begin{align}
    \overline Y&\prec \overline Z\, \mbox{ a.s.},\label{eq:compareYZ}
\end{align}
for $\overline Z$ defined by (\ref{eq:defZvector}). The solution is unique if
it holds that
\begin{equation}
\label{eq:condstabImp1}
\pr{\overline Z_1=0}>0.
\end{equation}
\end{proposition}

\begin{proposition}
\label{prop:stabunder}
The recursive equation
\begin{equation}
\label{eq:recurstatunder}
\underline Y\circ\theta=\underline\varphi(\underline Y),\,\mbox{ a.s.}
\end{equation}
admits at least one $\overline{\left(\R_+\right)^S}$-valued solution.
Furthermore, any such solution $\underline Y$ is such that
\begin{align*}
    \underline Y &\prec \underline Z\, \mbox{ a.s.},
\end{align*}
for $\underline Z$ defined by (\ref{eq:defunderZvector}), and the solution is unique if it holds that
\begin{equation}
\label{eq:condstabGamma}
\pr{\underline Z_1 = 0}>0.
\end{equation}
\end{proposition}

\subsection{Main results}
\label{subsec:stab}
We have the following results,
\begin{theorem}
\label{thm:stabImp1}  If it holds that
\begin{equation}
\pr{\left\{\overline Y(1)=0\right\}\,\bigcap\,\left(\bigcap_{\ell=2}^{S}\left\{\overline Y(\ell)\le
\sum_{i=0}^{\ell-2}\tau\circ\theta^i\right\}\right)}>0\label{eq:condstabImp2}
\end{equation}
for some solution $\ovY$ to (\ref{eq:recurstatover}), then
there exists a unique solution $W$ to (\ref{eq:recurstatImp}), that
is such that
\begin{equation}
\label{eq:compareWY} \underline Y \prec W \prec \overline Y \mbox{ a.s.,}
\end{equation}
for any solutions $\underline Y$ of (\ref{eq:recurstatunder}) and $\overline Y$ of (\ref{eq:recurstatover}). Consequently
the loss probability $\mathbb P_l$ of the system satisfies
\[\pr{\underline Y(1) > D} \le \mathbb P_l \le \pr{\overline Y(1)>D} \le \pr{\overline Z_S > D},\]
for any such $\ovY$.
\end{theorem}

The following (stronger, but explicit) sufficient stability condition and upper-bound follow,
\begin{corollary}
\label{cor:stabImp2} The conclusions of Theorem \ref{thm:stabImp1}
are valid if it holds that
\begin{equation}
\label{eq:condstabImp1} \pr{\overline Z_1=0}>0.
\end{equation}
Moreover any solution $W$ to (\ref{eq:recurstatImp}) satisfies
\begin{equation}
\label{eq:compareWZ} W \prec \overline Z \mbox{ a.s.}
\end{equation}
and thus
\[\mathbb P_l \le \pr{\overline Z_S > D}.\]
\end{corollary}

\begin{proof}
Plainly, in view of (\ref{eq:compareYZ}), as $\overline Z_1 \ge \overline Z_2 \ge ... \ge \overline Z_S$ a.s.,
(\ref{eq:condstabImp1}) entails (\ref{eq:condstabImp2}) and (\ref{eq:compareWZ}) readily follows from (\ref{eq:compareWY}).
\end{proof}

\paragraph{Weak stability}
Whenever condition (\ref{eq:condstabImp2}) is relaxed, weak stability of the system holds under the following lattice assumptions,
\begin{theorem}
\label{thm:extension}
Suppose that both $\sigma$ and $\xi$ take value in a set of the form
\begin{equation}
\label{eq:defLalpha}
\mathscr L_\alpha:=\left\{n\alpha;\,n\in \N\right\},
\end{equation}
then there exists a weak solution to (\ref{eq:recurstatImp}).
\end{theorem}


\paragraph{Independent case}
The stability conditions above can be made more explicit under the following specific independence assumptions on the input of the system:
\begin{align*}
\mathbf{(H_a)}:&\,\,\{\xi_n\}\mbox{ and }\{\sigma_n+D_n\}\mbox{ are IID and independent of one another};\\
& \\
\mathbf{(H_b)}:&\,\,\{\xi_n\},\,\{\sigma_n\}\mbox{ and }\{D_n\}\mbox{ are IID and independent of one another}\\
               &\,\,\mbox{ (GI/GI/S/S+GI system).}
\end{align*}
Clearly, $\mathbf{(H_b)}$ entails $\mathbf{(H_a)}$. We have the following result,
\begin{theorem}
\label{thm:GIGI} The conclusions of Theorem \ref{thm:stabImp1}
are valid:
\begin{itemize}
\item[(i)] for any system satisfying $\mathbf{(H_a)}$ and such that
\begin{equation}
\label{eq:condstabImp3} \pr{\sigma+D \le \tau}>0;
\end{equation}
\item[(ii)] for any system satisfying $\mathbf{(H_b)}$ and such that
\begin{equation}
\label{eq:condstabImp4} \pr{\sigma < \tau}>0.
\end{equation}
\end{itemize}
\end{theorem}
Theorems \ref{thm:stabImp1}, \ref{thm:extension} and \ref{thm:GIGI} are proven in Section \ref{sec:proofs}.

\section{Proofs}
\label{sec:proofs}
Let us first observe the following result,
\begin{lemma}
\label{lemma:compareImp} Let $\varphi$, $\overline{\varphi}$ and $\underline{\varphi}$be defined respectively
by (\ref{eq:coordSRSImp}), (\ref{eq:defoverphi}) and (\ref{eq:defunderphi}). Then, for any two elements $u$ and $v$ of
$\overline{\left(\R_+\right)^S}$ and any $i \in \llbracket
1,S \rrbracket$,
\begin{itemize}
\item[(i)]
$
\biggl[\forall j \in \llbracket i,S \rrbracket,\,u(j) \le
v(j)\biggl]\, \Longrightarrow \biggl[\forall j \in \llbracket
i,S \rrbracket,\,\varphi(u)(j) \prec \overline{\varphi}(v)(j)\mbox{ a.s.}\biggl].
$
\item[(ii)]
$\biggl[\forall j \in \llbracket i,S \rrbracket,\,u(j) \le
v(j)\biggl]\, \Longrightarrow \biggl[\forall j \in \llbracket
i,S \rrbracket,\,\underline{\varphi}(u)(j) \prec \varphi(v)(j)\mbox{ a.s.}\biggl].$
\end{itemize}
\end{lemma}

\begin{proof}
(i) First notice that 
\begin{multline}
u(1)+\sigma\ind_{\{u(1) \le D\}} =
\left(u(1)+\sigma\right)\ind_{\{u(1)\le D\}}+u(1)\ind_{\{u(1)>
D\}}\\
\le \left(\sigma+D\right)\ind_{\{u(1)\le D\}}+u(1)\ind_{\{u(1)>
D\}} \le u(1)\vee (\sigma+D).\label{eq:majorevarphi}
\end{multline}
Fix $i \in \llbracket 1,S \rrbracket$, and suppose that $u(j) \le
v(j)$ for all $j \in \llbracket i,S \rrbracket$. From
(\ref{eq:coordSRSImp}),we have that
\begin{align*}
\varphi(u)(S)  &\le \Bigl[u(S)\vee
\left(u(1)\vee\left(\sigma+D\right)\right) - \tau
\Bigl]^+\\
&=\Bigl[u(S)\vee \left(\sigma+D\right) - \tau \Bigl]^+
\le \Bigl[v(S)\vee \left(\sigma+D\right) - \tau \Bigl]^+
=\overline\varphi(v)(S)\mbox{ a.s.}
\end{align*}
and for $j \in \llbracket i,S-1 \rrbracket$,
\begin{align*}
\varphi(u)(j)  &\le \left[\Bigl(u(j)\vee
\left(u(1)\vee\left(\sigma+D\right)\right)\Bigl)\wedge\, u(j+1) - \tau
\right]^+\\
&= \Bigl[\left(u(j)\vee \left(\sigma+D\right)\right)\wedge u(j+1) - \tau \Bigl]^+\\
&\le \Bigl[\left(v(j)\vee \left(\sigma+D\right)\right)\wedge v(j+1) - \tau \Bigl]^+\\
&= \overline\varphi(v)(j)\mbox{ a.s..}
\end{align*}

\bigskip
\noindent (ii) We also have that
\begin{align}
v(1)+\sigma\ind_{\{v(1) \le D\}} &=
\left(v(1)+\sigma\right)\ind_{\{v(1)\le D\}}+v(1)\ind_{\{v(1)>
D\}}\nonumber\\
&\ge (v(1)\vee \sigma)\ind_{\{v(1)\le D\}}+(v(1)\vee D)\ind_{\{v(1)>
D\}}\nonumber\\
& \ge v(1)\vee (\sigma\wedge D).\label{eq:minorevarphi}
\end{align}
Thus, if $i \in \llbracket 1,S \rrbracket$, is such that $u(j) \le
v(j)$ for all $j \in \llbracket i,S \rrbracket$, we obtain that
\begin{align*}
\varphi(v)(S)  &\ge \Bigl[v(S)\vee
\left(v(1)\vee\left(\sigma\wedge D\right)\right) - \tau
\Bigl]^+\\
&=\Bigl[v(S)\vee \left(\sigma\wedge D\right) - \tau \Bigl]^+
\ge \Bigl[u(S)\vee \left(\sigma\wedge D\right) - \tau \Bigl]^+
=\underline\varphi(u)(S)\mbox{ a.s.}
\end{align*}
and for $j \in \llbracket i,S-1 \rrbracket$, that
\begin{align*}
\varphi(v)(j)  &\ge \left[\Bigl(v(j)\vee
\left(v(1)\vee\left(\sigma\wedge D\right)\right)\Bigl)\wedge\, v(j+1) - \tau
\right]^+\\
&=\Bigl[\left(v(j)\vee \left(\sigma\wedge D\right)\right)\wedge v(j+1) - \tau \Bigl]^+\\
&\ge \Bigl[\left(u(j)\vee \left(\sigma\wedge D\right)\right)\wedge u(j+1) - \tau \Bigl]^+\\
&=\underline\varphi(u)(j)\mbox{ a.s.,}
\end{align*}
which concludes the proof. 
\end{proof}

\begin{proof}[Proof of Theorem \ref{thm:stabImp1}]
Throughout this proof, fix a solution $\ovY$ to (\ref{eq:recurstatover}) and a solution
$\undY$ to (\ref{eq:recurstatunder}).
Define the family of $\overline{(\R_+)^S}$-valued
random variables
\[\mathscr K=\left\{X:\,X \prec \overline Y\mbox{ a.s.}\right\}\]
and the event
\begin{equation}
\maG:=\left\{\ovY(1)=0\right\}\,\bigcap\,\left(\bigcap_{\ell=2}^{S}\left\{\ovY(\ell)\le
\sum_{i=0}^{\ell-2}\tau\circ\theta^i\right\}\right).\label{eq:condstabImp2bis}
\end{equation}
Let $X\in \mathscr K$. Denote by $\left\{W^{[X]}_{n}\right\}$, the
workload sequence of the system, when setting the initial value as
$W^{[X]}_{0}=X$ a.s..
Then, assertion (i) of Lemma
\ref{lemma:compareImp} applied to $i=1$, implies by an immediate induction that
\begin{equation*}
W^{[X]}_{n} \prec \ovY\circ\theta^n,\,n\in\N,\mbox{
a.s..}\end{equation*}
Thus, for almost every sample on
the event $\theta^{-n}\maG$ we have that
\[\left\{\begin{array}{ll}
W^{[X]}_{n}(1)&=0;\\
W^{[X]}_{n}(\ell)&\le
\displaystyle\sum_{i=0}^{\ell-2}\tau\circ\theta^{n+i}  ;\,\ell \in
\llbracket 2,S \rrbracket.
\end{array}\right.\]
Therefore, on $\theta^{-n}\maG$,
\begin{itemize}
\item in the vector $W^{[X]}_{n+1}$, the coordinate corresponding to $W^{[X]}_{n}(2)$ vanishes and the one corresponding to
$W^{[X]}_{n}(1)$ becomes
$\left[\sigma\circ\theta^n-\tau\circ\theta^n\right]^+$ since a
customer is accepted by the corresponding server;
\item in $W^{[X]}_{n+2}$, the coordinate corresponding to $W^{[X]}_{n}(3)$ vanishes, the one corresponding to
$W^{[X]}_{n}(2)$ becomes
$\left[\sigma\circ\theta^{n+1}-\tau\circ\theta^{n+1}\right]^+$, and
the one corresponding to $W^{[X]}_{n}(1)$ equals\\
$\left[\sigma\circ\theta^n-\tau\circ\theta^n-\tau\circ\theta^{n+1}\right]^+$;
\item[$\vdots$]
\item[$\vdots$]
\medskip
\item the vector $W^{[X]}_{n+S-1}$ has at least $1$ null coordinate, and its $S-1$ last coordinates all are functions
of
$\left\{\left(\sigma\circ\theta^{n+i},\sigma\circ\theta^{n+i}\right),\,i\in
\llbracket 0,S-2 \rrbracket\right\}$.
\end{itemize}
Consequently, $\left\{\theta^{-n}\maG\right\}$ is a stationary
sequence of renovating events of length $S-1$ for any sequence
$\left\{W^{[X]}_{n}\right\}$ with $X \in \mathscr K$ (see \cite{Bor84}).
So (\ref{eq:condstabImp2bis}) entails the existence of a solution to
(\ref{eq:recurstatImp}), applying Theorem 4 in \cite{Foss92} (or
equivalently, Theorem 1 p.260 in \cite{Bor84} or Corollary 2.5.1 in
\cite{BacBre02}). Hereafter, we denote by $W$ such a solution.

We now prove that
\begin{equation}
\label{eq:compareYW0} W \prec \overline Y \mbox{ a.s..}
\end{equation}
Let us assume hereafter that $\pr{W \prec \mathbf 0}>0$ (which is necessarily the case if
$\pr{\sigma < \tau}>0$), otherwise (\ref{eq:compareYW0}) is trivial. Then, first observe that
\begin{equation}
\label{eq:compareYW1} \pr{\maD}:=\pr{W(S)
<W(1)+\sigma\ind_{\{W(1)\le D\}}}>0.
\end{equation}
Indeed, if we had $W(S) \ge W(1)+\sigma\ind_{\{W(1)\le D\}}$ a.s.,
we would have from (\ref{eq:coordSRSImp}) that a.s.
\[W(S)\circ\theta = [W(S)-\tau]^+,\]
in other words $W(S)=0$ a.s., an absurdity. Hence
(\ref{eq:compareYW1}). Let us also observe that
\begin{equation}
\label{eq:compareevents}\maD \subset \{W(1) \le D\}\end{equation}
since, for any sample in $\maD \cap \{W(1)>D\}$ we would have that
$W(S) < W(1)+\sigma\ind_{\{W(1)\le D\}}=W(1),$ an absurdity.

Let us now define the events
\[\maB_i:=\Bigl\{W(j) \le \overline Y(j)\,\mbox{ for all }j\in \llbracket i,S \rrbracket\Bigl\};\,\,i\in \llbracket 1,S \rrbracket.\]
From Assertion (i) of Lemma \ref{lemma:compareImp} we have for any $i$, for almost
every sample in $\maB_i$,
\[W(j)\circ\theta=\varphi(W)(j) \le \overline\varphi(\ovY)(j)=\ovY(j)\circ\theta\,\mbox{ for all }j\in \llbracket i,S \rrbracket,\]
in other words the events $\maB_i$, $i\in \llbracket 1,S
\rrbracket$, all are $\theta$-contracting. 
First, from (\ref{eq:compareevents}), on $\maD$ we have that
\begin{multline*}
W(S)\circ\theta=[W(1)+\sigma - \tau]^+\ind_{\{W(1)\le D\}}\\
\le[\sigma+D - \tau]^+\le [\ovY(S)\vee(\sigma+D) - \tau]^+=\ovY(S)\circ\theta,
\end{multline*} 
in other words $\maD \subset \theta^{-1}\maB_S.$ Thus with
(\ref{eq:compareYW1}),
\[\pr{\maB_S}=\pr{\theta^{-1}\maB_S}\ge \pr{\maD}>0,\]
hence the $\theta$-contracting event $\maB_S$ is almost sure.

If $S\ge 2$, suppose that $\maB_{i}$ is almost sure for some $i \in
\llbracket 2,S \rrbracket$. Then, from (\ref{eq:majorevarphi}), for
almost all samples on $\maD \cap \maB_i$ we have that
\begin{align*}
W(i-1)\circ\theta 
                  &= \left[\left(W(1)+\sigma\right)\ind_{\{W(1)\le D\}}\wedge W(i)-\tau\right]^+\\
                  &\le \left[\left(\sigma+D\right)\wedge W(i)-\tau\right]^+\\
                  &\le \left[\left(\sigma+D\right)\wedge \ovY(i)-\tau\right]^+\\
                  &\le \left[\left(\ovY(i-1)\vee\left(\sigma+D\right)\right)\wedge \ovY(i)-\tau\right]^+\\
                  &=\ovY(i-1)\circ\theta.
\end{align*}
In other words, $\left(\maD \cap \maB_i\right) \subset
\theta^{-1}\maB_{i-1}$, which entails that
\[\pr{\maB_{i-1}}=\pr{\theta^{-1}\maB_{i-1}} \ge \pr{\maD \cap \maB_i}>0,\]
so $\maB_{i-1}$ is almost sure since it is $\theta$-contracting. We
conclude by induction on $i$ that the events $\maB_{i}$, $i \in
\llbracket 2,S \rrbracket$ are all almost sure, which concludes the
proof of (\ref{eq:compareYW0}). 

\medskip

Let us now show that
\begin{equation}
\label{eq:compareunderYW}
\underline{Y} \prec W\mbox{ a.s..}
\end{equation}
We proceed similarly as for (\ref{eq:compareYW0}). Likewise, the assertion (ii) of Lemma \ref{lemma:compareImp}
shows that the events
\[\maC_i:=\Bigl\{\underline Y(j) \le W(j)\,\mbox{ for any }j\in \llbracket i,S \rrbracket\Bigl\};\,\,i\in \llbracket 1,S \rrbracket\]
are all $\theta$-contracting. Moreover, observe that the event
\[\maE:=\{\underline Y(S) \le \sigma\wedge D\}\]
is not negligible, because if it were so, then we would have $\underline Y(S)\circ\theta >0$ a.s. and in turn,
\[\undY(S)\circ\theta-\undY(S) =\tau \mbox{ a.s.},\]
another contradiction to the Ergodic Lemma. But on $\maE$, we have that
\[\undY(S)\circ\theta=\left[\sigma\wedge D -\tau\right]^+\le \left[(\sigma\wedge D)\vee W(S) -\tau\right]^+=W(S)\circ\theta,\]
so $\maE \subset \theta^{-1}\maC_S.$ In turn, the $\theta$-contracting event $\maC_S$ is not negligible by $\theta$-invariance, therefore it is
almost sure. Is $S\ge 2$, we conclude again by induction: if we suppose that $\maC_{i}$ is almost sure for some $i \in
\llbracket 2,S \rrbracket$, then for
almost all samples on $\maE \cap \maC_i$, in view of (\ref{eq:minorevarphi}) we have that
\begin{align*}
W(i-1)\circ\theta 
                  &= \left[\left(W(i-1)\vee\left(W(1)+\sigma\ind_{\{W(1)\le D\}}\right)\right)\wedge W(i)-\tau\right]^+\\
                  &\ge \left[\left(W(i-1)\vee\left(\sigma\wedge D\right)\right)\wedge W(i)-\tau\right]^+\\
                  &\ge \left[\left(\sigma\vee D\right)\wedge \underline Y(i)-\tau\right]^+\\
                  &= \left[\left(\underline Y(i-1)\vee\left(\sigma\wedge D\right)\right)\wedge \underline Y(i)-\tau\right]^+\\
                  &=\underline Y(i-1)\circ\theta.
\end{align*}
This shows that $\left(\maE \cap \maC_i\right) \subset
\theta^{-1}\maC_{i-1}$. We conclude that all events $\maC_{i},\,i\in \llbracket 2,S \rrbracket$ are almost sure exactly as for
the events $\maB_{i},\,i\in \llbracket 2,S \rrbracket$. Hence (\ref{eq:compareunderYW}).

\bigskip

Finally, (\ref{eq:compareYW0}) readily entails that
$\{\theta^{-n}\maG\}$ is a stationary sequence of renovating events
of length $S-1$, for any sequence
$\left\{W\circ\theta^n\right\},$ with $W$ a solution to
(\ref{eq:recurstatImp}). Remark 2.5.3 of \cite{BacBre02} then shows
the uniqueness of the solution of (\ref{eq:recurstatImp}). Also, it
follows from (\ref{eq:compareYW0}) that
\[\pr{\undY(1) \le D} \le \mathbb P_l = \pr{W(1)\le D} \le \pr{\ovY(1) \le D},\]
which completes the proof.
%
\end{proof}

We now turn to the proof of Theorem \ref{thm:extension},
\begin{proof}[Proof of Theorem \ref{thm:extension}]
The extension $\bar \Omega$ can be constructed explicitly using the
technique developed in \cite{Moy15}. For simplicity, we reproduce
the notation therein.

\medskip

Applying again assertion (i) of Lemma \ref{lemma:compareImp}, we see that
$\varphi(u) \prec \overline\varphi(u)$ a.s. for any $u\in (\R_+)^S$. Also, it is easily seen that the map
$\overline\varphi$ is a.s. $\prec$-non-decreasing, and the equation (\ref{eq:recurstatover}) admits a finite solution $\ovY$ in view of Proposition \ref{prop:stabover}.
 Consequently, the assumption (H1) of \cite{Moy15} is satisfied.
 So does condition (H3) of [\emph{ibid.}], as the locally finite set $\mathscr L_\alpha$ is clearly almost surely stable by the mapping $\varphi$.

Consequently, we are in the case (iii) of Theorem 3 of \cite{Moy15}: there exists a weak solution to (\ref{eq:recurstatImp}).
(See
the precise construction of the extension $\bar{\mathscr Q}=\left(\bar\Omega, \bar\maF, \bp, \bar\theta\right)$ of $\mathscr Q$ in
\cite{Moy15}.) In particular, denoting for all $x \in \N^S$ and all samples $\om$,
$$\Phi^n[\om](x):=W^{[x]}_{n}\left(\theta^{-n}\om\right)=\varphi[\theta^{-1}\om]\circ\varphi[\theta^{-2}\om]\circ ...\circ\varphi[\theta^{-n}\om]\left(x\right),$$
and defining the random set
\begin{equation}
\label{eq:defH1} H=\bigcap_{n \ge 1} \Phi^n\left(\mathscr L_\alpha\cap \lllbracket
\mathbf 0,Y\circ\theta^{-n} \rrrbracket \right)
\end{equation}
(recalling the notation (\ref{eq:lll})), the extension $\bar\Omega$ is defined by
$$\bar \Omega=\Bigl\{(\om,  x) \in \Omega\times \RS;   x\in H[\om]\Bigl\},$$
and the shift $\theta$ on $\bar\Omega$, by
$$\bar \theta[\om,x]=\left(\theta\om,\varphi[\om](x)\right),\,(\om,x)\in \bar\Omega.$$
It is then easily seen that the couple $(\bar W,\bar\varphi)$ defined by
\begin{equation}
\label{eq:defsolextension} \bar W[\omega,x]:=
x,\,\,\,\bar \psi[\omega,x]:=\varphi[\omega],
\end{equation}
readily satisfies $\bar W\circ\bar\theta=\bar \psi\left(\bar W\right),\bp-\mbox{a.s.}.$ Hence the result. 
\end{proof}

We conclude with the proof of Theorem \ref{thm:GIGI},
\begin{proof}[Proof of Theorem \ref{thm:GIGI}]
To prove assertion (i) of Theorem \ref{thm:GIGI}, it suffices to observe that the proof of
Corollary 2 in \cite{Moy13}, which shows the equivalence between (\ref{eq:condstabImp1}) and \eqref{eq:condstabImp3}
in the GI/GI/./.+GI case, also holds true in the more general case where $\mathbf{(H_a)}$ holds. Corollary \ref{cor:stabImp2} allows to conclude.

\medskip

We now turn to the proof of (ii). Consider a system satisfying $\mathbf{(H_b)}$. If (\ref{eq:condstabImp4}) holds true, then there exists
$x>0$ and $0<\epsilon<x$ such that $\pr{\sigma \le x-\epsilon}>0$ and $\pr{\tau > x}>0$.
Observe that there exists $n_0 \in \N$ such that
\[\pr{Z_1\circ\theta^{-n_0} < n_0\epsilon} = \pr{Z_1 < n_0\epsilon} >0.\]
Then, define the event
\[\maA = \left\{Z_1 \circ\theta^{-n_0} < n_0\epsilon\right\} \cap \left(\bigcap_{i=1}^{n_0} \{\tau\circ\theta^{-i}>x\}\right) \cap \left(\bigcap_{i=1}^{n_0} \{\sigma\circ\theta^{-i}\le x-\epsilon\}\right).\]
It follows from the independence assumption that the three events in the intersection above are independent of one another, and therefore
\begin{equation}
\label{eq:prA}
\pr{\maA}>0.
\end{equation}
Denote for any $\ell \in \llbracket 1,S \rrbracket$, by {\em server $\ell$}
the server whose workload at the arrival time of customer $-n_0$ corresponds to the coordinate $W_{-n_0}(\ell)$.
Also, denote for all such $\ell$, by $\maC_\ell \subset \{1,2,...,n_0\}$, the (random) set gathering the opposites of all indexes of the customers arrived between customer $C_{-n}$ and $C_{-1}$ included, and taken care of by server $\ell$. Then, for any $\ell \in \llbracket 1,S \rrbracket$, we are in the following alternative,
\begin{enumerate}
\item[(1)] if server $\ell$ never idles before the arrival time of customer $-1$, then its workload at time 0 equals
\begin{multline}\label{eq:case1}
\left[W_{-n_0}(\ell)+\sum_{i \in \maC_\ell} \sigma\circ\theta^{-i} - \sum_{i=1}^{n_0} \tau\circ\theta^{-i} \right]^+\\
  \le \left[W_{-n_0}(\ell)+\sum_{i=1}^{n_0} \left(\sigma\circ\theta^{-i} - \tau\circ\theta^{-i}\right) \right]^+.\end{multline}
\item[(2)] if for some $k_\ell \in \llbracket 1,n_0 \rrbracket$ (take the smallest one), server $\ell$ is idling upon the arrival of $C_{-k_\ell}$, then its workload
at time 0 equals
\begin{equation}
\label{eq:case2}
\left[\sum_{i \in \maC_\ell \cap \llbracket 1,k_\ell \rrbracket} \sigma\circ\theta^{-i} - \sum_{i=1}^{k_\ell} \tau\circ\theta^{-i} \right]^+
  \le \left[\sum_{i=1}^{k_\ell} \left(\sigma\circ\theta^{-i} - \tau\circ\theta^{-i}\right) \right]^+.
  \end{equation}
\end{enumerate}
Let us now define the random set
\begin{equation}
\label{eq:defH2} H=\bigcap_{n \ge 1}H^n:=\bigcap_{n \ge 1} \Phi^n\left(\left[
\mathbf 0,Y\circ\theta^{-n} \right] \right).
\end{equation}
Let $\om \in \maA$ and suppose that $W_{-n_0}[\om]=x \in \left[\mathbf 0,Y\left[\theta^{-n_0}\om\right]\right]$.
Then, for any $j\in \llbracket 1,S \rrbracket$, let $\ell_j$ be the index of the server whose workload at 0 is given by $W_0(j)$. Then, either
 $\ell_j$ is in the alternative (1) above, in which case from (\ref{eq:case1}) together with (\ref{eq:compareYZ}),
\begin{align*}
\left(\Phi^{n_0}[\om](x)\right)(j)=W_0[\om](j)&\le  \left[x(\ell_j)+\sum_{i=1}^{n_0} \left(\sigma\circ\theta^{-i}[\om] - \tau\circ\theta^{-i}[\om]\right) \right]^+\\
&\le \left[Y_S\left[\theta^{-n_0}\om\right]+\sum_{i=1}^{n_0} \left(\sigma\circ\theta^{-i}[\om] - \tau\circ\theta^{-i}[\om]\right) \right]^+\\
&\le \left[n_0\epsilon - n_0(x-\epsilon-x)\right]^+=0,
\end{align*}
or $\ell_j$ is in alternative (2), in which case from (\ref{eq:case2}),
\[\left(\Phi^{n_0}[\om](x)\right)(j)\le \left[k_{\ell_j}(x-\epsilon-x) \right]^+=0.\]

In all cases, we obtain $\Phi^{n_0}[\om](x)=\mathbf 0$. Since this is true for all $x \in \left[\mathbf 0,Y\left[\theta^{-n_0}\om\right]\right]$,
we have $H^{n_0}[\om]=\{\mathbf 0\}$.
We conclude that $\maA \subset \{H^{n_0}\mbox{ is finite }\}$, and thus with (\ref{eq:prA}), that
\[\pr{\bigcup_{n\ge 1}\{H^{n}\mbox{ is finite }\}}>0.\]
Thus assumption (8) in \cite{Moy15} is satisfied, so we can apply Theorem 1 in [{\em ibid.}]: there exists a weak solution to (\ref{eq:recurstatImp}). Moreover,
in view of the argument above we have that Card$\,H=1$ (where $H$ is defined by (\ref{eq:defH2})), so there exists a unique solution on the original quadruple $\mathscr Q$, in view of Theorem 2 of \cite{Moy15}.
\end{proof}

\providecommand{\bysame}{\leavevmode\hbox to3em{\hrulefill}\thinspace}
\providecommand{\MR}{\relax\ifhmode\unskip\space\fi MR }
\providecommand{\MRhref}[2]{%
  \href{http://www.ams.org/mathscinet-getitem?mr=#1}{#2}
}
\providecommand{\href}[2]{#2}

\end{document}